\documentclass[a4paper,11pt]{amsart}
\usepackage{amsmath,inputenc,euscript,amssymb,geometry}
\usepackage{graphicx}
\usepackage{amssymb}
\usepackage{latexsym}
\usepackage{amssymb,amsbsy,amsmath,amsfonts,amssymb,amscd,color,ulem}
\usepackage{hyperref}
 \usepackage{stmaryrd}
 \usepackage{array}	
 \usepackage{cancel}

\allowdisplaybreaks[1]

\newtheorem{lemma}{Lemma}[section]
\newtheorem{theo}[lemma]{Theorem}

\newtheorem{corol}[lemma]{Corollary}

\begin{document}
\title[On self-similar singularity formation for the binormal flow]{On self-similar singularity formation for the binormal flow}
\author[A. Gu\'erin]{Anatole Gu\'erin}
    \address[A. Gu\'erin]{Universit\'e Paris-Saclay, Institut de Math\'ematiques d'Orsay (IMO), F-91405, France and Sorbonne Universit\'e, Laboratoire Jacques-Louis Lions (LJLL), F-75005 Paris, France} 
\email{anatole.guerin@universite-paris-saclay.fr}
\date\today
\maketitle

\begin{abstract}
The aim of this article is to establish a concise proof for a stability result of self-similar solutions of the binormal flow, in some more restrictive cases than in \cite{banica2015initial}. This equation, also known as the Local Induction Approximation, is a standard model for vortex filament dynamics, and its self-similar solution  describes the formation of a corner singularity on the filament. Our approach strongly uses the link that Hasimoto pointed out in 1972 between the solution of the binormal flow and the one of the 1-D cubic Schrödinger equation, as well as the existence results associated to the latter. 
\end{abstract}

\section{Introduction}

In this paper, we propose a new proof of the stability of self-similar solutions of the binormal flow
\begin{equation} \label{BF}
\chi_t=\chi_x \wedge \chi_{xx}.
\end{equation}
In terms of physics, $\chi(t,x)$ belongs to $\mathbb R^3$, $t$ represents the time and $x$ is the arclength variable. This equation was proposed in 1906 by DaRios in \cite{darios1906motion} and re-discovered in 1965 by Arms and Harma in \cite{arms1965localized}, for modeling a vortex filament dynamic under Euler equations.\\ 
In a few words, its formal derivation goes as follows. If we consider the velocity of an incompressible fluid $u$ and its vorticity $\omega$, the Biot-Savart law tells us that:
\[
u(t,x)=\displaystyle\int_{\mathbb{R}^3}\frac{(x-y)\wedge\omega(t,y)}{4\pi|x-y|^3}dy.
\]
Then, if we suppose that $\omega(t)$ belongs to a $1$D curve (i.e. $\omega=\Gamma\chi_x\delta_\chi$) with $\chi_x$ of norm $1$, we can write:
\[
u(t,x)=\displaystyle\int_{-\infty}^\infty\frac{(x-\chi(t,s))\wedge\omega(t,\chi(t,s))}{4\pi|x-\chi(t,s)|^3}ds.
\]
Conducting a Taylor expansion around zero on the space variable and restricting the domain of integration to $[-L,L]$ approximates the previous integral by:
\begin{align*}
u(t,0) \approx&\frac{\Gamma}{4\pi}\displaystyle\int_{-L}^L\frac{((x_1,x_2,0)-s\chi_s(t,0)-\frac{s^2}{2}\chi_{ss}(t,0))\wedge(\chi_s(t,0)+s\chi_{ss}(t,0))}{|(x_1,x_2,-s)|^3}ds\\
=&\frac{\Gamma}{4\pi}\frac{(-x_2,x_1,0)}{\epsilon^2}\displaystyle\int_{-\frac{L}{\epsilon}}^{\frac{L}{\epsilon}}\frac{ds}{(1+s^2)^\frac{3}{2}}+\frac{\Gamma}{4\pi}(x_1,x_2,0)\wedge\chi_{ss}(t,0)\displaystyle\int_{-L}^L\frac{s}{|\epsilon^2+s^2|^\frac{3}{2}}ds\\
&-\frac{\Gamma}{8\pi}\chi_s(t,0)\wedge\chi_{ss}(t,0)\displaystyle\int_{-\frac{L}{\epsilon}}^{\frac{L}{\epsilon}}\frac{s^2}{|1+s^2|^\frac{3}{2}}ds.
\end{align*}
The first term corresponds to a fluid rotating around a still vertical axis, the second term vanishes by a parity argument, and the third term gives us \eqref{BF}, after a time-renormalization. This model is sometimes called the Local Induction Approximation (LIA) or vortex filament equation (VFE), and is the subject of further discutions  in \cite{batchelor1967intro} , \cite{ricca1996contributions}  and more recently by Jerrard and Seis in \cite{jerrard2017vortex} with stronger assumptions but rigorous arguments.\\

In 1972, Hasimoto linked the solutions $\chi(t,x)$ of \eqref{BF} to solutions of a 1-D cubic Schrödinger equation by using the Frenet and parallel frames in \cite{hasimoto1972soliton} . This transformation is in the same spirit as the Mandelung transform. \\
Conversly, for a given real potential $a$ and a given solution $\psi$ of
\begin{equation}\label{NLS}
i\psi_t+\psi_{xx} +\frac12(|\psi|^2-a(t))\psi=0,
\end{equation}
the Hasimoto transformation is reversible by using Frenet frames for non vanishing curvatures vortices. However the calculations are much faster and work for any curvatures by constructing first parallel frames $(T,e_1,e_2)(t,x)$ that satisfy:
\begin{equation} \label{frame}
T_x=\mathcal{R}(\overline\psi N),\quad N_x=-\psi T,\quad T_t=\mathcal{I}(\overline{\psi_x}N), \quad N_t=-i\psi_x T-\frac i2(|\psi|^2-a(t)) N,
\end{equation}
with $N=e_1+ie_2$, and any orthonormal basis as initial data. It follows that the vector $T$ satisfies the 1-D Schrödinger map with values in $\mathbb S^2$:
\[
T_t=T\wedge T_{xx},
\]
and can be integrated into a solution $\chi$ of the binormal flow \eqref{BF} starting at a point $P$ at $(t_0,x_0)$ with the formula:
\[
\chi(t,x)=P+\int_{t_0}^t(T\wedge T_x) (\tau,x_0)d\tau+\int_{x_0}^xT(t,s)ds,\quad \forall (t,x).
\]

In this paper, we study the stability of the self-similar solutions $\{ \chi_\alpha \}_{\alpha>0}$ of \eqref{BF} determined for $t>0$ by a curvature of $\frac\alpha{\sqrt t}$ and a torsion of $\frac x{2t}$.The behaviour of $\chi_\alpha(t,s)$ for $t > 0$ was exhibited by physicists in \cite{LD} and \cite{LRT} and a numeric study on it was done in \cite{BU1}.  In \cite{Gut3}, it has been proven that they are solutions of \eqref{NLS}, smooth as long as $t>0$ and have a trace at $t=0$ forming a one corner polygonal line of angle $\theta$ such that 
\begin{equation}\label{angle}
\sin \frac\theta2=e^{-\pi\frac{\alpha^2}2}.
\end{equation}
This class of solutions correspond to solutions of 1-D cubic NLS solutions 
\[
\psi_\alpha(t,x)=\alpha\frac{e^{i\frac{x^2}{4t}}}{\sqrt t},
\]
 taking $a(t)=\frac{\alpha^2}t$ in \eqref{NLS}.
\begin{theo}[The initial value problem for the binormal flow]\label{maintheo}
Let $\chi_0$ a smooth arc-length parametrized curve of $\mathbb R^3$, except at one point located at arc-length $x=0$ where it forms a corner of angle $\theta$. 
Let $c$ be the curvature of $\chi_0$, $\tau$ its torsion and $\alpha$ given by \eqref{angle}. \\
If $\alpha$ defined from $\theta$ by \eqref{angle} is small enough, and if 
\[
c\in W^{3,1}\cap H^2,\quad \frac cx\in W^{2,1}\cap H^2, \quad x^2c\in W^{3,1}\cap H^2,\quad (1+x^2)c\in L^2, \quad x^{-2}c\in L^2, 
\]
\[
\tau \in H^2 \quad \text{and} \quad \tau^2\in H^1,
\] 
then there exists $t_0>0$ and
\begin{equation}
\chi(t,x) \in \mathcal C([-t_0,t_0],Lip)\cap\mathcal C([-t_0,t_0] \backslash \{ 0\} , \mathcal C^4),
\end{equation}
a solution of the binormal flow \eqref{BF} on $(0,t_0]$, having $\chi_0$ as a limit at time $t=0$, and there exists $C>0$ such that:
\begin{equation} \label{convchi}
\sup_x|\chi(t,x)-\chi_0(x)|\leq C\sqrt t.
\end{equation}
Moreover, the tangent vector $T=\partial_x \chi$ has a limit at time zero with the same time-decay rate:
\begin{equation} \label{convT}
\forall t>0 \quad \forall x\in \mathbb R \quad \exists C(x) \quad |T(t,x)-\partial_x\chi_0(t)|\leq C(x) t^\frac14.
\end{equation}
\end{theo}
This type of result has already been proven by Banica and Vega in Theorem 1.2 of \cite{banica2015initial}, under weaker assumptions on the curvature and torsion of $\chi_0$. As a counterpart, the corresponding scattering results for \eqref{NLS} (existence of wave operator and asymptotic completeness) obtained in \cite{banica2012scattering} are with weaker decay. As a consequence, the proof  require to obtain asymptotic space states for $T(t,x)$ and $N(t,x)$ when $x\rightarrow \pm\infty$, and a much more technical iterative argument to obtain the limit for $T$ and $N$ at time $t=0$.\\
 In here, we will use stronger decay of the wave operator results in \cite{BV1} to give a concise proof of Theroem \ref{maintheo}.\par
 We note that even under more restrictive hypothesis than in \cite{banica2015initial}, we do not have an asymptotic completeness result with better decay, that would allow us to give also a concise proof of Theorem 1.3 of the second stability result in \cite{banica2015initial}.

Let us streamline here the constructive proof of Theorem \ref{maintheo}. Denoting $T_0$ the tangent vector to $\chi_0$, we define the complex valued functions $g\in \mathbb C$ and $N_0 \in\mathbb S^2+i\mathbb S^2 $ defined by the parallel frame system: 
\begin{equation}\label{ci}
\left \{ 
\begin{array}{rl}
  T_{0x}(x)=&\Re(g(x) N_0(x))  \\
 N_{0x}(x)= &-\overline g(x)T_0(x)  \\
\end{array}
 \right. ,
\end{equation}
with initial data $(A^+_{\alpha},B^+_{\alpha})$ for $x>0$ and $(A^-_{\alpha},B^-_{\alpha})$ for $x<0$, where $A^\pm_{\alpha}$ and $B^\pm_{\alpha}$ stand for the complex vectors appearing in the asymptotics of the normals vectors of the same self-similar solution $\chi_\alpha$ (see Theorem 1 of \cite{Gut3}).\\
Let us note that, using Frenet frame, there exists $\gamma\in[0,2\pi]$ such that:
\begin{equation} \label{uct}
g(x) = c(x)e^{i(\int_0^x\tau(s)ds+\gamma)},
\end{equation}
  as explained in Remark 2.1 of \cite{banica13}.\par
Now set:
\begin{equation}\label{defup}
u_+= \mathcal F^{-1} \sqrt i \left(g(2\cdot)e^{i\alpha^2 \log|\cdot|} \right).
\end{equation}
The hypothesis of Theorem \ref{maintheo} on $c$ and $\tau$ allow $u_+$ to belong to some particular Sobolev spaces in order to use the existence of a wave operator for \eqref{NLS} proved in Theorem 1.4 of \cite{BV1}. More precisely, $u_+$ is in  $\dot H^{-2} \cap H^2 \cap W^{2,1}$ and $\alpha$ is small, so there exists $t_0>0$ and a unique solution of \eqref{NLS} on $(0,t_0]$ of the form:
\begin{equation} \label{solnls}
\psi(t,x)=\frac{e^{i\frac{x^2}{4t}}}{\sqrt t}\left(\alpha+ \overline u\left(\frac1t,\frac xt\right) \right),
\end{equation}
with $u$ being a perturbation that writes:
\begin{equation}\label{defu}
u(t,x)=e^{it\partial_x^2} u_+(x) + r(t,x).
\end{equation}
The proof of this result uses scattering methods after performing a pseudo-conformal transformation, and allows us to have the following control on the time decay of the remainder term $r$, for $k=1$ and $k=2$:
\begin{equation}\label{borner}
\|r(t)\|_{L^2_x} =\mathcal O(t^{-\frac12}) \quad,\quad\|\nabla^kr(t)\|_{L^2_x}=\mathcal O(t^{-1}).
\end{equation}\par
The next step in our proof is to use the parallel frame \eqref{frame} with the function $\psi$ given by \eqref{solnls} to construct a solution $\chi$ of \eqref{BF} on $(0,t_0]$.\par
Then, we consider the vectors $T$ and $N$ given by \eqref{frame}, as well as $\tilde N$ a modulated version of $N$ defined later. We prove in section \ref{limit} that $T$ and $\tilde N$ admit a trace at time $t=0$, thanks to bounds on the perturbation $u$ given in Corollary \ref{bornu}, consequence of bound \eqref{borner}.\par
Then, in section \ref{moreinfo} we find the ODE system verified by $T_{|t=0}$ and $\tilde N_{|t=0}$ for $x\neq 0$ that turns out to be the same as the one of $T_0$ and $N_0$, due to \eqref{defup}. Sections \ref{limit} and \ref{moreinfo} are the part of the proof that simplify consistently the proof in \cite{banica2015initial}. \par
Finally, in section \ref{description}, we use self-similar paths to determine $T_{|t=0}$ and $\tilde N_{|t=0}$ at $x=0^+$ and $x=0^-$ for the ODE system, that coincides with the corner singularity directions of $\chi_0$ and complete the Cauchy Problem. These last results allows us to conclude in section \ref{fin} that we recovered $\chi_0$ at time $t=0$.

\section{Construction of perturbed self-similar solution of the binormal flow}
As announced in the introduction, we first define the complex-valued function $g$ with the system verified by $\chi_0$'s tangent and normal vectors $T_0$ and $N_0$:
\begin{equation}
\left \{ 
\begin{array}{rl}
  T_{0x}(x)=&\Re(g(x) N_0(x))  \\
 N_{0x}(x)= &-\overline g(x)T_0(x)  
\end{array}
 \right. ,
\end{equation}
with initial data $(A^+_{\alpha},B^+_{\alpha})$ for $x>0$ and $(A^-_{\alpha},B^-_{\alpha})$ for $x<0$,
and consider
\begin{equation} \label{up}
u_+= \mathcal F^{-1} \sqrt i \left(g(2\cdot)e^{i\alpha^2 \log|\cdot|} \right).
\end{equation}
We now deduce regularity on $u_+$ from the the hypothesis of Theorem \ref{maintheo} on $c$ and $\tau$, which is the purpose of the following lemma.
\begin{lemma}\label{ctau}
Consider the curvature $c$ and the torsion $\tau$ of a parametrized curve. Define $u_+$ by formula \eqref{up} and recall expression \eqref{uct} of $g$.\\
If 
\[
c\in W^{3,1}\cap H^2,\quad \frac cx\in W^{2,1}\cap H^2, \quad x^2c\in W^{3,1}\cap H^2,\quad (1+x^2)c\in L^2, \quad x^{-2}c\in L^2, 
\]
and
\[
\tau \in H^2 \quad \text{and} \quad \tau^2\in H^1,
\]
then
\[
u_+\in W^{1,2}\cap H^2\cap \dot H^{-2} \quad \text{and} \quad (1+x^2)u_+ \in L^\infty , \quad (1+x^2)xu_+\in L^\infty.
\]
\end{lemma}
This lemma will allow us to apply a wave operator existence theorem right after, but also to use the weighted $L^\infty$ bound on $u_+$ in the proof of Corollary \ref{bornu}.
\begin{proof}
The idea of the proof is to write the inverse Fourier transform formula and perform integration by parts on it, to gain decay. We have by definition:
 \[
 u_+(x) = \int_{\mathbb R} e^{-ixy} \sqrt i c(2y) e^{i(\int_0^{2y} \tau(s)ds+\gamma)} e^{i\alpha^2\log|y|}dy,
 \]
 so integrating by parts to times leads to:
 \begin{align*}
  u_+(x) =&- \int_{\mathbb R} \frac{e^{-ixy}}{-ix} \sqrt i (2c'(2y)+ ic(2y)\tau(2y) + i\alpha^2 \frac{c(2y)}y)e^{i(\int_0^{2y} \tau(s)ds+\gamma)} e^{i\alpha^2\log|y|}dy\\
  =&\int_{\mathbb R} \frac{e^{-ixy}}{x^2} \sqrt i (4c''(2y)+ i2c'(2y)\tau(2y) + i2c(2y)\tau'(2y))e^{i(\int_0^{2y} \tau(s)ds+\gamma)} e^{i\alpha^2\log|y|}dy\\
  &+\int_{\mathbb R}\frac{e^{-ixy}}{x^2} \sqrt i   i\alpha^2 \frac{2yc'(2y)+c(2y)}{y^2}e^{i(\int_0^{2y} \tau(s)ds+\gamma)} e^{i\alpha^2\log|y|}dy\\
  &+ \int_{\mathbb R} \frac{e^{-ixy}}{x^2} \sqrt ii \tau(2y) (2c'(2y)+ ic(2y)\tau(2y) + i\alpha^2 \frac{c(2y)}y)e^{i(\int_0^{2y} \tau(s)ds+\gamma)} e^{i\alpha^2\log|y|}dy\\
&+ \int_{\mathbb R} \frac{e^{-ixy}}{x^2} \sqrt i \frac{i\alpha^2}{y}(2c'(2y)+ ic(2y)\tau(2y) + i\alpha^2 \frac{c(2y)}y)e^{i(\int_0^{2y} \tau(s)ds+\gamma)} e^{i\alpha^2\log|y|}dy.
 \end{align*}
 Because all of the terms in those integrals are by hypothesis either $L^1$, or a product of two $L^2$ functions, it all converges and we deduce that $u_+\in L^1$ and $(1+x^2)u_+ \in L^\infty$.\\
 Then, it is straightforward to check that $(1+x^2)xu_+\in L^\infty$ with an additional integration by parts.
To obtain $\nabla u_+\in L^1$, we write :
 \[
 \nabla u_+(x) = -i\int_{\mathbb R} e^{-ixy} y\sqrt i c(2y) e^{i(\int_0^{2y} \tau(s)ds+\gamma)} e^{i\alpha^2\log|y|}dy,
 \]
 and perform as well two integration by parts. We similarly show that $\nabla^2 u_+\in L^1$.\\
 Finally, for the $L^2$ hypothesis, we use Parseval identity to claim that $(1+x^2)c\in L^2$ and $x^{-2}c\in L^2$ imply that $u_+\in H^2\cap\dot H^{-2}$.
 \end{proof}
Thanks to this lemma we have that $u_+$ is in $W^{1,2}\cap H^2\cap \dot H^{-2}$ under the hypothesis of Theorem \ref{maintheo}. Therefore, we can apply Theorem 1.2 of \cite{BV1}, to obtain a unique solution of \eqref{NLS} on $(0,t_0]$ that writes:
\begin{equation} 
\psi(t,x)=\frac{e^{i\frac{x^2}{4t}}}{\sqrt t}\left(\alpha+ \overline u\left(\frac1t,\frac xt\right) \right),
\end{equation}
where:
\begin{equation} \label{constru}
u(t,x)=e^{it\partial_x^2} u_+(x) + r(t,x),
\end{equation}
with $r$ satisfying \eqref{borner}.\par
Then, equations \eqref{frame} of Hasimoto's construction allows us to construct $\chi$, a solution of \eqref{BF} on $(0,t_0]$ by its tangent and normal vectors $T$ and $N$. However, in order to identify the trace of $\chi(t)$ at time $t=0$, we need a better understanding of the perturbation $u$. 

\subsection{Preliminary bound}\leavevmode\\
In order to obtain a bound on $u$ that is sharp enough, we shall use the decay given by \eqref{borner}.
\begin{corol}[$L^\infty$ bound on the perturbation $u$] \label{bornu}
Let $u$ defined by \eqref{constru}. Under the hypothesis of Theorem \ref{maintheo}, we have the following bound on $u$ and its derivative as $t$ goes to zero:
\[
\left| u \left(\frac1t, \frac xt \right) \right| \leq  t^{\frac12},  \quad \text{ with } \quad \left| r\left( \frac 1t,\frac xt\right) \right| \leq t^{\frac34},
\]
and 
\[
\left| \partial_x u \left(\frac1t, \frac xt \right) \right| \leq  \frac x{\sqrt t}  + t^{\frac12}, \quad \text{ with }\quad \left| \partial_x r\left( \frac 1t,\frac xt\right) \right| \leq t^{\frac12}.
\]
Moreover, we have:
\begin{equation} \label{cancel}
\left| \frac {ix}{2t} u \left(\frac1t, \frac xt \right) - \left[ u \left(\frac1t, \frac xt \right)\right]_x \right| \leq t^\frac12.
\end{equation} 
\end{corol}
The last estimate comes from a cancellation, and gives us more decay that expected. 
\begin{proof}
First, we give a bound of the remainder term $r$ and its derivative using the decay \eqref{borner} given in Theorem 1.2 of \cite{BV1} (wave operator existence). For this, we apply the Gagliardo Niremberg interpolation inequality:
\[
\left| r\left( \frac 1t,\frac xt\right) \right| \leq t^{\frac14} \left\| r\left( \frac 1t,\cdot \right) \right\|^{\frac12}_{L^2} t^{-\frac14} \left\| \partial_x r\left( \frac 1t,\cdot \right) \right\|^{\frac12}_{L^2} \leq t^{\frac34},
\]
and similarly:
\[
\left| \partial_x r\left( \frac 1t,\frac xt\right) \right| \leq t^{\frac12}.
\]
Next, we simply write:
\[
\left| e^{i\frac1t\partial_x^2}u_+\left(\frac xt \right) \right| = \left| \int \sqrt t e^{i\frac t4 (\frac xt -y)^2}u_+(y)dy \right| \leq \sqrt t \|u_+\|_{L^1},
\]
and for the other term we use the fact that $xu_+(x)\in L^1$, obtained in Lemma \ref{ctau}:
\[
\left| \partial_x e^{i\frac1t\partial_x^2}u_+\left(\frac xt \right) \right| =  \left| \partial_x \int \sqrt t e^{i\frac t4 (\frac xt -y)^2}u_+(y)dy \right| = \left| \frac {ix}{2t} e^{i\frac1t\partial_x^2}u_+\left(\frac xt \right)  \right| +  \left|  \int \sqrt t e^{i\frac t4 (\frac xt -y)^2} \frac{iy}{2} u_+(y)dy \right| , \\
\]
that ensures:
\[
\left| \partial_x u \left(\frac1t, \frac xt \right) \right| \leq  \frac x{\sqrt t} + \sqrt t .
\]
Finally, \eqref{cancel} comes directly from the previous expression, as we write:
\[
\left| \frac {ix}{2t} u \left(\frac1t, \frac xt \right) - \left[ u \left(\frac1t, \frac xt \right)\right]_x \right|=  \left|  \int \sqrt t e^{i\frac t4 (\frac xt -y)^2} \frac{iy}{2} u_+(y)dy  \right|.
\]
\end{proof}
We are now ready to tackle our proof.
\subsection{Limit at time $t=0$} \label{limit} \leavevmode\\
As announced, the next step is to prove the existence of a limit for vectors $T$ and $N$, up to a phase. 
\begin{lemma}[Limit of vector T]\label{T0}
The tangent vector $T$ of $\chi$ has a limit at time zero with a convergence rate given by:
\[
\forall t_0\geq t_2\geq t_1>0 \quad \forall x\in \mathbb R^* \quad |T(t_2,x)-T(t_1,x)| \leq xt_2^{\frac14}  + t_2^{\frac 34} + \frac {\sqrt{t_2}}x.
\]
\end{lemma}
This lemma, gives us the convergence rate \eqref{convT} announced in Theorem \ref{maintheo}.
\begin{proof}
Now let $t_2\geq t_1 >0$,
\begin{align*}
|T(t_2,x)-T(t_1,x)|=& \left| \int_{t_1}^{t_2} T_t(t,x)dt \right| =  \left| \Im \int_{t_1}^{t_2} \overline{{\psi }_x}N(t,x)dt \right|\\
=&  \left| \Im \int_{t_1}^{t_2} \frac{e^{-i\frac{x^2}{4t}}}{\sqrt t}\left(\frac{-ix}{2t}  u\left(\frac1t,\frac xt \right) -i\frac{x\overline {\alpha }}{2t} +  \left[  u\left(\frac1t,\frac xt \right) \right]_x \right) N(t,x) dt \right| \\
\leq & x t_2^{\frac14}   + t_2  +\left| \Im \int_{t_1}^{t_2} e^{-i\frac{x^2}{4t}} \frac{ix\alpha}{2t\sqrt t} N(t,x)dt \right | \\
&+ \left| \Im  \int_{t_1}^{t_2} \frac{e^{-i\frac{x^2}{4t}}}{\sqrt t}\left( \frac{-ix}{2t} e^{i\frac1t\partial_x^2}u_+\left(\frac xt \right) + \left[ e^{i\frac1t\partial_x^2}u_+\left(\frac xt \right) \right]_x  \right) N(t,x)dt \right|   ,
\end{align*}
where the terms with the remainder $r$ has provided enough decay. Then, if we use \eqref{cancel}, we have that:
\[
\left| \Im  \int_{t_1}^{t_2} \frac{e^{-i\frac{x^2}{4t}}}{\sqrt t}\left( \frac{-ix}{2t} e^{i\frac1t\partial_x^2}u_+\left(\frac xt \right) + \left[ e^{i\frac1t\partial_x^2}u_+\left(\frac xt \right) \right]_x  \right) N(t,x)dt \right|   \leq t_2.
\]
For the other term, we integrate by parts:
\begin{align*}
\left| \Im \int_{t_1}^{t_2} e^{-i\frac{x^2}{4t}} \frac{ix\alpha}{2t\sqrt t} N(t,x)dt \right|\leq  & \left| \Im  \left[ e^{-i\frac{x^2}{4t}} \frac{2\sqrt t \alpha}{x} N(t,x) \right]_{t_1}^{t_2} \right| + \left| \Im \int_{t_1}^{t_2} e^{-i\frac{x^2}{4t}} \frac{ \alpha}{x\sqrt t} N(t,x) dt \right| \\
&+\left | \Im \int_{t_1}^{t_2} e^{-i\frac{x^2}{4t}} \frac{2\sqrt t \alpha}{x} N_t(t,x) dt \right| \\
\leq & \frac {2\alpha\sqrt {t_2}}x  +  \left| \Im \int_{t_1}^{t_2} e^{-i\frac{x^2}{4t}} \frac{2\sqrt t \alpha}{x} N_t(t,x) dt \right|.
\end{align*}
We must now expand the term in $N_t$: 
\begin{align*}
&\left| \Im \int_{t_1}^{t_2} e^{-i\frac{x^2}{4t}} \frac{2\sqrt t \alpha}{x} N_t(t,x) dt \right| \\
\leq & \left| \Im \int_{t_1}^{t_2} e^{-i\frac{x^2}{4t}} \frac{2\sqrt t \alpha}{x} i \frac{e^{i\frac{x^2}{4t}}}{\sqrt t}\left(\frac{ix}{2t} \overline u\left(\frac1t,\frac xt \right) +i\frac{x\alpha }{2t} +\left[ \overline u\left(\frac1t,\frac xt \right) \right]_x \right) T(t,x) dt \right| \\
&+ \left|  \Im \int_{t_1}^{t_2} e^{-i\frac{x^2}{4t}} \frac{2\sqrt t \alpha}{x} \frac i2 \left( \frac{|u\left(\frac1t,\frac xt \right) |^2}t +\frac{2\Re(u\left(\frac1t,\frac xt \right) \alpha )}t  \right) N(t,x) dt \right|\\
\leq &   t_2^{\frac34} + \frac{t_2}x ,
\end{align*}
using both \eqref{cancel} and the fact that $T$ is real, so we have: 
\[
\Im \int_{t_1}^{t_2} \frac {\alpha^2}{t}T(t,x)dt =0 .
\]
To sum up, we showed that:
\[
\forall t_0\geq  t_2\geq t_1>0 \quad \forall x\in \mathbb R^* \quad |T(t_2,x)-T(t_1,x) | \leq xt_2^{\frac14}  + t_2^{\frac 34} + \frac {\sqrt{t_2}}x ,
\]
and the lemma is proven.\\
Note that, for self similar paths, we also obtained that $|T(t,x\sqrt t)-T(t,x\sqrt t)| $ goes to zero as $t$, $\frac1x$ and $x\sqrt t$ simultaneously go to zero.
\end{proof} 
In order for $N$ to converge, we must add a phase.
\begin{lemma}[Limit of vector N]\label{N0}
 Let us write
\[
\tilde N(t,x) = e^{i\alpha^2\ln \frac {|x|}{\sqrt t}}N(t,x) = e^{i\phi} N,
\]
where $N$ is the normal vector of $\chi$. Then $\tilde N$ has a limit at time zero with a convergence rate given by:
\[
\forall t_0 \geq t_2\geq t_1>0 \quad \forall x\in \mathbb R^* \quad |\tilde N(t_2,x)-\tilde N(t_1,x)| \leq xt_2^\frac14+ t_2^\frac12 + \frac {\sqrt {t_2}}x +\frac{t_2}{x^2}  .
\]
\end{lemma}
Note that the factor $|x|$ in $\phi$ could be replaced by anything independent of $t$, but is chosen for assuring properties at time $t=0$ as we will see in Lemma \ref{CIN}.
\begin{proof}
To follow the proof, the reader must only keep in mind that $|u\left(\frac1t,\frac xt \right)|$ behaves at worse like $ t^{\frac34} $ and $|\partial_x u\left(\frac1t,\frac xt \right) |$ at worse like $\sqrt t + t x\sqrt t$.\\
Recalling that:
\[
{\tilde N}_t = e^{i\phi} N_t -i\frac {\alpha^2}{2t} N(t,x) e^{i\phi},
\]
given $0<t_1\leq t_2\leq t_0$, we have:
\begin{align*}
&\tilde N(t_2,x)-\tilde N(t_1,x) \\
=& \int_{t_1}^{t_2} {\tilde N}_t(t,x)dt = \int_{t_1}^{t_2}  -i{\psi }_xTe^{i\phi} + \frac i2 (|\psi |^2-\frac {\alpha^2}t)Ne^{i\phi}dt-i\frac {\alpha^2}{2t} N(t,x) e^{i\phi}\\
=&-\int_{t_1}^{t_2}i \frac{e^{i\frac{x^2}{4t}}}{\sqrt t}\left(\frac{ix}{2t} \overline u\left(\frac1t,\frac xt \right) +i\frac{x\alpha }{2t} +\left[ \overline u\left(\frac1t,\frac xt \right) \right]_x \right) T(t,x)e^{i\phi}dt\\
&+\frac i2\int_{t_1}^{t_2} \left( \cancel {\frac{\alpha^2}t} + \frac{|u\left(\frac1t,\frac xt \right) |^2}t +\frac{2\Re(u\left(\frac1t,\frac xt \right) \alpha )}t -\cancel {\frac{\alpha^2}t} \right) N(t,x)e^{i\phi}dt.\\
&- \int_{t_1}^{t_2}  i\frac {\alpha^2}{2t} N(t,x) e^{i\phi} dt.
\end{align*}
As before, we use \eqref{cancel} so terms with $u$ in the first integral partially cancel with each other. Using bounds of Corollary \ref{bornu}, we are now left with only a difference to study:
\[
| \tilde N(t_2,x)-\tilde N(t_1,x)  | \leq  xt_2^\frac14+ t_2^\frac12 +
 \left| - \int_{t_1}^{t_2} i \frac{e^{i\frac{x^2}{4t}}}{\sqrt t} \frac {ix}{2t} \alpha  T(t,x) e^{i\phi}dt - \int_{t_1}^{t_2}  i\frac {\alpha^2}{2t} N(t,x) e^{i\phi} dt \right|.
 \]
For that, we integrate by parts the first term:
\begin{align*}
&\int_{t_1}^{t_2} \frac{e^{i\frac{x^2}{4t}}x}{2t\sqrt t} \alpha  T(t,x)e^{i\phi}dt\\
=&\left[e^{i\frac{x^2}{4t}}\frac{2\sqrt t}{ix}\alpha T(t,x) e^{i\phi} \right]_{t_1}^{t_2} - \int_{t_1}^{t_2} e^{i\frac{x^2}{4t}}\frac{1}{ix\sqrt t} \alpha T(t,x)  e^{i\phi}dt -\int_{t_1}^{t_2} e^{i\frac{x^2}{4t}}\frac{2\sqrt t}{ix} \alpha T_t(t,x)e^{i\phi} dt \\
&+\int_{t_1}^{t_2} e^{i\frac{x^2}{4t}}\frac{\alpha^2 }{x\sqrt t}\alpha T(t,x) e^{i\phi}  dt,
\end{align*}
and get:
\[
| \tilde N(t_2,x)-\tilde N(t_1,x)  | \leq  xt_2^\frac14+ t_2^\frac12 + \frac {\sqrt {t_2}}x + \left| \int_{t_1}^{t_2} e^{i\frac{x^2}{4t}}\frac{2\sqrt t}{ix} \alpha T_t(t,x) e^{i\phi} dt  -  \int_{t_1}^{t_2}  i\frac {\alpha^2}{2t} N(t,x) e^{i\phi} dt \right|.
\]
We then use the fact that $T_t=\Im (\overline{{\psi }_x}N) = \frac 1{2i} (\overline{{\psi }_x}N - {\psi }_x \overline N )$ to write:
\begin{align*}
&\int_{t_1}^{t_2} e^{i\frac{x^2}{4t}}\frac{2\sqrt t}{ix} \alpha T_t(t,x) e^{i\phi} dt \\
= &\frac 1{2i} \int_{t_1}^{t_2} \frac{2 }{ix}  \alpha  \left(\frac{-ix}{2t}  u\left(\frac1t,\frac xt \right) -i\frac{x\alpha}{2t} +\left[  u\left(\frac1t,\frac xt \right) \right]_x \right)  N(t,x)e^{i\phi} dt \\
&- \frac 1{2i} \int_{t_1}^{t_2} e^{i\frac{x^2}{4t}}\frac{2}{ix} \alpha   e^{i\frac{x^2}{4t}}\left(\frac{ix}{2t} \overline u\left(\frac1t,\frac xt \right) +i\frac{x\alpha }{2t} +\left[ \overline u\left(\frac1t,\frac xt \right) \right]_x \right) \overline N(t,x) e^{i\phi}dt. \\
\end{align*}
Again, thanks to Corollary \ref{bornu}, only the terms without $u$ are worth studying. Moreover, the first term cancels with the term coming from the phase $\phi$. Therefore we have:
\[
| \tilde N(t_2,x)-\tilde N(t_1,x)  | \leq  xt_2^\frac14+ t_2^\frac12 + \frac {\sqrt {t_2}}x + \left| \frac1{2i}\int_{t_1}^{t_2}e^{i\frac{2x^2}{4t}}\frac {\alpha ^2}t \overline N(t,x)e^{i\phi}dt  \right|.
\]
The other one has a phase, so we perform a second integration by parts on it:
\begin{align*}
&\frac1{2i}\int_{t_1}^{t_2}e^{i\frac{2x^2}{4t}}\frac {\alpha ^2}t\overline N(t,x)e^{i\phi}dt \\
=&\frac 1{2i}\left[e^{i\frac{2x^2}{4t}} \frac{2\alpha ^2 t}{ix^2} \overline N(t,x) e^{i\phi} \right]_{t_1}^{t_2} +\frac1{2i}\int_{t_1}^{t_2} e^{i\frac{2x^2}{4t}} \frac{2\alpha ^2 }{ix^2}\overline N(t,x) e^{i\phi} dt \\
&-\frac1{2i}\int_{t_1}^{t_2}  e^{i\frac{2x^2}{4t}} \frac{2\alpha ^2 t}{ix^2} \overline N_t(t,x) e^{i\phi} dt + \frac1{2i} \int_{t_1}^{t_2} e^{i\frac{2x^2}{4t}} \frac{\alpha^2\alpha ^2 }{x^2} \overline N(t,x) e^{i\phi} dt.
\end{align*}
We finally expand the $N_t$ term and observe that it has the desired behavior:
\begin{align*}
&-\frac1{2i}\int_{t_1}^{t_2}  e^{i\frac{2x^2}{4t}} \frac{2\alpha ^2 t}{ix^2} \overline N_t(t,x) e^{i\phi}  dt\\
=&+\frac1{2i}\int_{t_1}^{t_2}  e^{i\frac{x^2}{4t}} \frac{2\alpha ^2 t}{ix^2} \frac i{\sqrt t}\left(\frac{-ix}{2t}  u\left(\frac1t,\frac xt \right) -i\frac{x\overline \alpha }{2t} -\left[  u\left(\frac1t,\frac xt \right) \right]_x \right) \overline T(t,x) e^{i\phi} dt\\
&-\frac1{2i}\int_{t_1}^{t_2}  e^{i\frac{2x^2}{4t}} \frac{2\alpha ^2 t}{ix^2} \left( \frac{|u\left(\frac1t,\frac xt \right) |^2}t +\frac{2\Re(\overline u\left(\frac1t,\frac xt \right) \overline \alpha )}t \right) N(t,x)e^{i\phi} dt.
\end{align*}
To sum up, we proved that:
\[
\forall t_0 \geq t_2\geq t_1>0 \quad \forall x\in \mathbb R^* \quad |\tilde N(t_2,x)-\tilde N(t_1,x)|\leq xt_2^\frac14+ t_2^\frac12 + \frac {\sqrt {t_2}}x  + \frac{t_2}{x^2}.
\]
As for $T$, we also obtained that, for self similar paths, $|N(t,x\sqrt t)-N(t,x\sqrt t)| $ goes to zero as $t$, $\frac1x$ and $x\sqrt t$ simultaneously go to zero.
\end{proof}
\subsection{More information about the tangents vectors at time $t=0$}\label{moreinfo} \leavevmode\\

The aim of this section is to quantify the evolution of $T_{|t=0}$ and $\tilde N_{|t=0}$ with respect to the space variable. More precisely, we will show that:
\[
\left\{ 
\begin{array}{cc}
T_x(0,x)=  &\Re  \frac1{\sqrt i}\widehat{u_+}\left(\frac x2\right) e^{-i \alpha^2 \log |x|}\tilde N(0,x), \\
 \tilde N_x(0,x)=  &  - \frac1{\sqrt i}\overline{\widehat{u_+}\left(\frac x2\right) e^{-i \alpha^2\log |x|}} T(0,x),
\end{array}
\right. \forall x\neq 0.
\]
Those two claims can be proved separately and that is what we are going to do.
\begin{lemma}[Properties of $T_{|t=0}$] \label{CIT}
Let $x\in \mathbb R^*$, then we have:
\[
T_x(0,x)=\lim_{t\rightarrow 0} T_x(t,x)= \Re  \frac1{\sqrt i}\widehat{u_+}\left(\frac x2\right) e^{-i \alpha^2 \log |x|}\tilde N(0,x) .
\]
\end{lemma}
\begin{proof}
Let  $(x_1,x_2)\in \mathbb {R_+^*}^2$. We are going to write the variation of $T$ at $t>0$ between $x_1$ and $x_2$, with the idea to make $t$ go to zero:
\begin{align*}
T(t,x_2)-T(t,x_1) =& \int_{x_1}^{x_2} T_x(t,s)ds = \int_{x_1}^{x_2}\Re(\overline \psi N)(t,s)ds \\
=& \Re \int_{x_1}^{x_2} \frac{e^{-i\frac{s^2}{4t}}}{\sqrt t} (u\left(\frac1t,\frac st\right) +\alpha) N(t,s)ds \\
=&\Re  \left[ e^{-i\frac{s^2}{4t}}\frac{2\sqrt t}{is} \alpha N(t,s)  \right]_{x_1}^{x_2} + \Re \int_{x_1}^{x_2}  e^{-i\frac{s^2}{4t}}\frac{2\sqrt t}{is^2} \alpha N(t,s) ds \\
&+ \Re \int_{x_1}^{x_2}  e^{-i\frac{s^2}{4t}}\frac{2}{is^2}e^{i\frac{s^2}{4t}}\alpha^2 T(t,s)ds+ \Re \int_{x_1}^{x_2} e^{-i\frac{s^2}{4t}}\frac{2}{is^2}e^{i\frac{s^2}{4t}} \overline{u}\left(\frac1t,\frac st\right)T(t,s)ds\\
&+ \Re \int_{x_1}^{x_2} \frac{e^{-i\frac{s^2}{4t}}}{\sqrt t} u\left(\frac1t,\frac st\right)N(t,s)ds.
\end{align*}
The last term will provide us the differential equation that we are looking for. The term in $\alpha^2$ vanishes since it is an imaginary term inside the $\Re$ operator. All the other termes go to zero with $t$ thanks to Corollary \ref{bornu}. \\
Now, recall that $u\left(\frac1t,\frac xt\right) = e^{i\frac1t\partial^2_x} u_+\left(\frac xt\right) + r\left(\frac1t,\frac xt\right) $.  If we write:
\[
 e^{i\frac 1t \partial_x^2}u_+\left(\frac xt \right) =  \frac{e^{i\frac{x^2}{4t}}}{\sqrt{\frac it}}\int e^{-i\frac{xy}2} e^{i\frac{y^2}4t}u_+(y)dy,
\]
we have:
\[
 \frac{e^{-i\frac{x^2}{4t}}}{\sqrt t}u\left(\frac1t,\frac xt \right)  =  \frac{1}{\sqrt i} \int e^{-i\frac{xy}2} e^{i\frac{y^2}4t}u_+(y)dy + \frac{e^{-i\frac{x^2}{4t}}}{\sqrt t} r\left(\frac1t,\frac xt\right) \underset{t\to 0}{\longrightarrow}  \frac1{\sqrt i}\widehat{u_+}\left(\frac x2\right) ,
\]
since $\|r\left(\frac1t,\frac xt\right)\|_{L^\infty} \leq t^{\frac34}$. Note that in \cite{banica2015initial}, $r$ decays like $t^{\frac14}$ so the present argument is not enough.\\
Then, let us consider $(t_n)_{n\in\mathbb Z}$ such that $\forall n\in\mathbb N,\quad e^{i \alpha^2\log \sqrt {t_n}}=1$ and $t_n \underset{n\to \infty}{\longrightarrow} 0$,
\[
 N(t_n,x) =e^{-i\phi(t_n,x)}\tilde N(t_n,x)  = e^{-i \alpha^2 \log \frac{|x|}{\sqrt t_n}}\tilde N(t_n,x) \underset{n\to \infty}{\longrightarrow} e^{-i \alpha^2 \log |x|}\tilde N(0,x),
\]
so by multiplying the limits:
\[
\Re  \frac{e^{-i\frac{x^2}{4t_n}}}{\sqrt {t_n}}u\left(\frac1{t_n},\frac x{t_n} \right)  N(t_n,x)  \underset{n\to \infty}{\longrightarrow}  \Re \frac1{\sqrt i}\widehat{u_+}\left(\frac x2\right) e^{-i \alpha^2 \log |x|}\tilde N(0,x),
\]
and by dominated convergence:
\[ 
\Re \int_{x_1}^{x_2} e^{-i\frac{s^2}{4t_n}} \frac{u\left(\frac1{t_n},\frac x{t_n}\right) }{\sqrt {t_n}} N(t_n,s)ds \underset{n\to \infty}{\longrightarrow}  \Re \int_{x_1}^{x_2} \frac1{\sqrt i}\widehat{u_+}\left(\frac x2\right) e^{-i \alpha^2 \log |x|}\tilde N(0,x) .
\] 
To sum up, we proved that:
\[
T(t_n,x_2)-T(t_n,x_1) \underset{n\to \infty}{\longrightarrow}  \Re \int_{x_1}^{x_2} \frac1{\sqrt i}\widehat{u_+}\left(\frac x2\right) e^{-i \alpha^2 \log |x|}\tilde N(0,x)dx,
\]
and the conclusion of the lemma is obtained by taking $x_1=x$, $x_2=x+h$, dividing by $h$, using Lemma \ref{T0} and chosing $n$ large with respect to $h$. 
\end{proof}
\begin{lemma}[Properties of $ {\tilde N}_{|t=0}$]\label{CIN}
For $x \neq0$, we have:
\[
\tilde N_x(0,x)=\lim_{t\rightarrow 0} \tilde N_x(t,x) = \frac1{\sqrt i}\overline{\widehat{u_+}\left(\frac s2\right) e^{-i \alpha^2\log |x|} }T(0,s).
\]
\end{lemma}

\begin{proof}
Let $(x_1,x_2)\in\mathbb {R_+^*}^2$, we write:
\begin{align*}
 \tilde N(t,x_2)-  \tilde N (t,x_1)=& \int_{x_1}^{x_2}  {\tilde N}_x(t,s)ds = \int_{x_1}^{x_2} (-\psi T +i \frac{\alpha^2}{s} N)e^{i\phi} ds.
\end{align*}
The term produced by the phase will help removing an otherwise non vanishing term, so we start by looking at the integral of $N_x$:
\begin{small}
\begin{align*}
\int_{x_1}^{x_2}\psi(t,s)T(t,s)e^{i\phi}ds = & \int_{x_1}^{x_2}  \frac{e^{i\frac{s^2}{4t}}}{\sqrt t} \alpha  T(t,s) e^{i\phi}ds +\int_{x_1}^{x_2}\frac{ e^{i\frac{s^2}{4t}}}{\sqrt t} \overline u\left(\frac1t,\frac st\right)e^{i\phi} T(t,s)ds\\
=&\left[ e^{i\frac{s^2}{4t}}\frac{2\sqrt t}{is}\alpha  T(t,s) e^{i\phi} \right]_{x_1}^{x_2} +  \int_{x_1}^{x_2}    e^{i\frac{s^2}{4t}}\frac{2\sqrt t}{is^2} \alpha  T(t,s) e^{i\phi} ds \\
&- \int_{x_1}^{x_2}   e^{i\frac{s^2}{4t}}\frac{2\sqrt t}{is}   \alpha  T(t,s) i \frac{\alpha^2}{s} e^{i\phi}ds   -  \int_{x_1}^{x_2}    e^{i\frac{s^2}{4t}}\frac{2\sqrt t}{is}   \alpha T_s(t,s)e^{i\phi} ds\\
&+\int_{x_1}^{x_2}\frac{ e^{i\frac{s^2}{4t}}}{\sqrt t} \overline u\left(\frac1t,\frac xt\right) T(t,s)e^{i\phi}ds.
\end{align*}
\end{small}
As with $T$, we will treat the term with $u$ at the end, first we have to make sure that the $T_s$ term goes to zero with $t$, using that $T_s=\Re(\overline \psi N)$:

\begin{align*}
& \int_{x_1}^{x_2}    e^{i\frac{s^2}{4t}}\frac{2\sqrt t}{is}   \alpha T_s(t,s)e^{i\phi} ds= \int_{x_1}^{x_2}    e^{i\frac{s^2}{4t}}\frac{2\sqrt t}{is}   \alpha  \frac{e^{i\frac{s^2}{4t}}}{\sqrt t} \overline u\left(\frac1t,\frac st\right) \overline N(t,s) e^{i\phi}ds\\
&+\int_{x_1}^{x_2}  e^{2i\frac{s^2}{4t}}\frac{1}{is}   \alpha ^2e^{i\phi} \overline N(t,s)ds +  \int_{x_1}^{x_2}    \frac{ 1}{is}   \alpha^2   N(t,s) e^{i\phi}ds\\
&+\int_{x_1}^{x_2}    e^{i\frac{s^2}{4t}}\frac{1}{is}   \alpha  e^{-i\frac{s^2}{4t}}  u\left(\frac1t,\frac st\right)  N(t,s) e^{i\phi}ds.\\
 \end{align*}
The first term is treated with Cauchy Schwarz, as well as the fourth. The second tends to zero with an IBP and the third  is canceled by the phase.\\
We shall now obtain the differential equation verified by $\tilde N$. Again, using $(t_n)_{n\in\mathbb Z}$ such that $e^{i \alpha^2\log \sqrt {t_n}}=1$ and $t_n \underset{n\to \infty}{\longrightarrow} 0$,
\[
e^{i\phi(t_n,x) }T(t_n,x)  \underset{n\to \infty}{\longrightarrow} e^{i \alpha^2\log |x|} T(0,x),
\]
and by multiplying the limits under the integral we write:
\[
\int_{x_1}^{x_2}\frac{ e^{i\frac{s^2}{4t_n}}}{\sqrt {t_n}} \overline u\left(\frac1{t_n},\frac s{t_n}\right)e^{i\phi} T(t,s) ds \underset{n\to \infty}{\longrightarrow} \int_{x_1}^{x_2} \frac1{\sqrt i}\overline{\widehat{u_+}}\left(\frac s2\right) e^{i \alpha^2\log |x|} T(0,s).
\]
Hence:
\[
 \tilde N (t_n,x_2)- \tilde N(t_n,x_1) \underset{n\to \infty}{\longrightarrow} -\int_{x_1}^{x_2} \frac1{\sqrt i}\overline{\widehat{u_+}\left(\frac s2\right) e^{-i \alpha^2\log |s|}} T(0,s)ds,
\]
and the conclusion of the lemma is obtained by taking $x_1=x$, $x_2=x+h$, dividing by $h$, using Lemma \ref{N0} and chosing $n$ large with respect to $h$. 
\end{proof}
\subsection{Description of the angles via self-similar paths} \label{description}\leavevmode\\
For the description of the angles, we will follow the same proof as for Proposition 5.1 of \cite{banica13}. For sake of completeness, we recall here the proof. As recalled in the introduction,  we denote by $A^\pm_{\alpha}\in\mathbb S^2$ the directions of the corner generated at time $t=0$ by the canonical self-similar solution $\chi_{\alpha}(t,x)$ of the binormal flow of curvature $\frac{\alpha}{\sqrt t}$:
\[
A^\pm_{\alpha}:=\partial_x \chi_{\alpha}(0,0^\pm).
\]
The frame of the profile $\chi(1)$ satisfies the system:
\begin{equation}\label{coin}
\left\{ 
\begin{array}{cc}
 \partial_xT_{\alpha}(1,x)= &  \Re(|\alpha  e^{-i\frac{x^2}4}N_{\alpha}(1,x)),    \\
 \partial_xN_{\alpha}(1,x)= &   -\alpha e^{i\frac{x^2}4}T_{\alpha}(1,x),
\end{array}
\right.
\end{equation}
and for $x\rightarrow\pm\infty$, there exists $B^\pm_{\alpha} \perp A^\pm_{\alpha}$, with $\Re (B^\pm_{\alpha}),\Im(A^\pm_{\alpha})\in \mathbb S^2$ such that:
\[
T_{\alpha}(1,x)=A^\pm_{\alpha} + \mathcal O(\frac1x) \quad \text{and}\quad e^{i\alpha^2\log|x|}N_{\alpha}(1,x)=B^\pm_{\alpha}+\mathcal O(\frac 1x).
\] 
\begin{lemma}[Self similar paths ]\label{auto}
Let $t_n$ be a sequence of positive times converging to zero. Up to a subsequence, there exists for all $x\in\mathbb R$ a limit given by:
\[
(T_*(x),N_*(x)) =\lim_{t\rightarrow 0} (T(t_n,x\sqrt{t_n}), N(t_n, x\sqrt{t_n})),
\]
such that $(T_*, N_*(x))$ satisfies system \eqref{coin} in the strong sense.\\
Then, there exists a unique rotation $\Theta $, such that, for $x\rightarrow \pm\infty$:
\[
T_*(x)=\Theta  (A^\pm_{\alpha})+\mathcal O(\frac1{|x|}),\quad N_*(x)=\Theta (B^\pm_{\alpha})+\mathcal O(\frac1{|x|}).
\]
\end{lemma}
\begin{proof}
Let $(t_n)_{n\in\mathbb N}\in \mathbb R_+^{\mathbb N}$ a sequence of positive times converging to $0$. As explained in \cite{banica13}, $u\in L^4((1,\infty),L^\infty)$ so we can chose $(t_n)_{n\in\mathbb N}$ such that $\|u(1/t_n)\|_{L^\infty}$ goes to zero.\\
We now naturally define the following sequences:
\[
\forall n\in \mathbb N \quad (T_n,N_n) = (T(t_n,x\sqrt{t_n}), N(t_n, x\sqrt{t_n})).
\]
Since $\|T\|_{L^\infty}\leq 1$ and $\|N\|_{L^\infty}\leq 2$ it is obvious that those sequences are bounded. Let us prove their equicontinuity.\\
For all $n\in \mathbb N,\quad T_n$ is derivable and using that $T_x = \Re(\overline\psi N)$ and $N_x=-\psi T$,
\[
T_n'(x) = \sqrt {t_n} \Re (\overline \psi N)(t_n, x\sqrt{t_n})=  \Re[ \alpha e^{-i\frac{x^2}4}N(t_n,x\sqrt{t_n})] + o(1)N_n(x).
\]
Similarly, for all $x\in\mathbb R$,
\[
N_n'(x)=  \sqrt {t_n}  (-\psi N)(t_n, x\sqrt{t_n})  =-\alpha e^{i\frac{x^2}4}T(t_n,x\sqrt{t_n})+ o(1)T_n(x).
\]
Sequences $(T_n', N_n')$ are uniformly bounded, so $(T_n,N_n)$ are equicontinuous.\\
By d'Arzela-Ascoli theorem on $\mathcal T = \{T_n,n\in\mathbb N\}$ and $\mathcal N =\{N_n,n\in\mathbb N\}$, there exists a subsequence of $(T_n,N_n)$, converging toward $(T_*(x),N_*(x))$. For convenience, we will not write the extractice.\\
As the coefficients involved in the ODE are analytic, we conclude that $(T_*,  N_*(x))$ satisfies system \eqref{coin} in the strong sense, as $(T_{\alpha}(x),N_{\alpha}(x))$.\\
Therefore, there exists an unique rotation $\Theta$ such that 
\[
\left\{\begin{array}{rl} 
T_*(x)=&\Theta(T_{\alpha}(x)),\\ 
\Re(N_*(x))=&\Theta(\Re(N_{\alpha}(x))),\\  
 \Im(N_*(x))=&\Theta(\Im(N_{\alpha}(x))). \end{array} \right. 
\] 
So we conclude that for $x\rightarrow \pm\infty$:
\[
T_*(x)=\Theta (A^\pm_{\alpha})+\mathcal O(\frac1{|x|}),\quad N_*(x)=\Theta(B^\pm_{\alpha})+\mathcal O(\frac1{|x|}).
\]
\end{proof}
\begin{lemma}[Description of the singularity]
We have 
\[
T(0,0^\pm)=\Theta(A_{\alpha}^\pm) \quad\text{and}\quad  e^{i\alpha^2\log|x|} \tilde N (0,0^\pm) = \Theta(B_{\alpha}^\pm),
\]
where $\Theta$ has been introduced in Lemma \ref{auto}.
\end{lemma}
The proof of this lemma uses all we did in the previous section concerning the limit of vectors $\tilde N$ and $T$.
\begin{proof}
Let $\varepsilon>0$. The main idea of this proof is to write 
\begin{align*}
|T(0,0^+)-\Theta(A_{\alpha}^+)| \leq &|T(0,0^+) - T(0,x\sqrt{t_n})|+|T(0,x\sqrt{t_n}) - T(t_n,x\sqrt{t_n})|\\
&+|T(t_n,x\sqrt{t_n}) - T_*(x)| + |T_*(x)-\Theta(A_{\alpha}^+)|.
\end{align*} 
First, we chose $x$ big enough, such that $|T_*(x)-\Theta(A_{\alpha}^+)| \leq \frac{\varepsilon}4$, thanks to Lemma \ref{auto}. \\
Then we chose $n$ big enough, such that $|T(t_n,x\sqrt{t_n}) - T_*(x)| \leq \frac{\varepsilon}4$ thanks to convergence, such that $|T(0,x\sqrt{t_n}) - T(t_n,x\sqrt{t_n})|  \leq \frac{\varepsilon}4$ thanks to Lemma \ref{T0} and finally such that $|T(0,0^+) - T(0,x\sqrt{t_n})| \leq \frac{\varepsilon}4$, using Lemma \ref{CIT}:
\[
|T(0,0^+) - T(0,x\sqrt{t_n})| \leq \|T_x\|_\infty x\sqrt {t_n} \leq C(u_+)x\sqrt {t_n}.
\]
So we have $|T(0,0^+)-\Theta(A_{\alpha}^+)| \leq  \varepsilon$, i.e.
\[
T(0,0^+)=\Theta(A_{\alpha}^+).
\]
Similarly, for $x<0$ we prove that $T(0,0^-)=\Theta(A_{\alpha}^-)$.\par
For $\tilde N$ we follow the same path, taking care to handle the phases. 
For $(t_n)_{n\in\mathbb N} \in \mathbb R_+^{\mathbb N}$ converging to zero, such that 
\[
\exp(i\alpha^2\log\sqrt{t_n})=1,
\]
we have:
\begin{align*}
&|\Theta(B^+_{\alpha})-   \tilde N (0,k+)|\\
 \leq &|\Theta(B^+_{\alpha}) - e^{i\alpha^2\log|x|} N_*(x)| +| e^{i\alpha^2\log|x|} N_*(x)  - e^{i\alpha^2\log|x|} N(t_n,x\sqrt t_n)|\\
&+|e^{i\alpha^2\log|x|} N(t_n, x\sqrt t_n)  - e^{i\alpha^2 \ln \frac {|x\sqrt t_n|}{\sqrt t_n}} N(t_n,x\sqrt t_n) |+|   \tilde N (t_n,x\sqrt t_n) - \tilde N (0,x\sqrt t_n) | \\
&+ |  \tilde N (0,x\sqrt t_n) -  \tilde N (0,k+)|.
\end{align*}
The first term is small for $x$ big enough thanks to Lemma \ref{auto}. The second is small for $n$ big enough thanks to Lemma \ref{auto}. The third term is zero, the fourth term is small when $t_n$ is small enough using Lemma \ref{N0}. Finally, the last term is controlled by $C(u)x\sqrt t_n$ due to Lemma \ref{CIN}, and we have the desired result.
\end{proof}
\section{Recovering the initial curve $\chi_0$} \label{fin}
In this section, we prove that the curve $\chi$ is equal to $\chi_0$ at time zero, combining the results of the two previous parts and the choice of $u_+$ in the introduction.\\
 The system that verify  $N$ and $T$ at time zero is the following:
\[
\left\{ 
\begin{array}{cc}
T_x(0,x)=  &\Re  \frac1{\sqrt i}\widehat{u_+}\left(\frac x2\right) e^{-i \alpha^2 \log |x|}\tilde N(0,x), \\
 \tilde N_x(0,x)=  &  - \frac1{\sqrt i}\overline{\widehat{u_+}\left(\frac x2\right) e^{-i \alpha^2\log |x|}} T(0,x),
\end{array}
\right. \forall x\neq 0,
\]
with initial value given by
\[
T(0,0^\pm)=\Theta(A_{\alpha}^\pm) \quad\text{and}\quad  e^{i\alpha^2\log|x|} \tilde N (0,0^\pm) = \Theta(B_{\alpha}^\pm).
\]
Recalling the definition of $u_+$ given by \eqref{defup}, $T(0)$ and $\tilde N(0)$ satisfy the same Cauchy system \eqref{ci} as $T_0$ and $N_0$, hence $\chi(0)=\chi_0$.

Finally, we are left to prove the convergence rate \eqref{convchi} of $\chi(t,x)$ as $t$ goes to zero. Since $\chi_t(t,x)=c(t,x)$ and $c(t,x)=|\psi(t,x)|\leq \frac C{\sqrt t}$, we have:
\[
|\chi(t_2,x)-\chi(t_1,x)|\leq \int_{t_1}^{t_2} \frac C{\sqrt t}dt \leq C \sqrt t_2,
\]
and Theorem \ref{maintheo} is proven.

\section*{Acknowledgments}
This paper has been written during my PhD under the supervision of Valeria Banica and Nicolas Burq, I would like to thank them for their precious help and discussions.

\bibliographystyle{aaai-named}

\end{document}